\documentclass[a4paper]{amsart}

\usepackage{epsfig}
\usepackage{amsthm,amsfonts}
\usepackage{amssymb,graphicx,color}
\usepackage[all]{xy}
\usepackage{enumerate}
\usepackage{verbatim}
\usepackage{hyperref}

\newtheorem{theorem}{Theorem}[section]
\newtheorem*{theorem*}{Theorem}
\newtheorem{lemma}[theorem]{Lemma}
\newtheorem{corollary}[theorem]{Corollary}
\newtheorem{proposition}[theorem]{Proposition}
\newtheorem{definition}[theorem]{Definition}

\newtheorem{example}[theorem]{Example}

\newtheorem{observation}[theorem]{Observation}

\DeclareMathOperator{\Div}{div}
\DeclareMathOperator{\grad}{grad}
\DeclareMathOperator{\ric}{Ric}

\begin{document}

\title[Spacelike Foliations on Lorentz manifolds]{Spacelike Foliations on Lorentz manifolds}

\author[Aldir Brasil]{Aldir Brasil}
\address[Aldir Brasil]{Universidade Federal do Ceará, Campus do Pici, CEP 60455-760, Fortaleza, Cear\' a, Brazil}
\email{aldir@mat.ufc.br}

\author[Sharief Deshmukh]{Sharief Deshmukh}
\address[Sharief Deshmukh]{King Saud University,	Riyadh 11451, Saudi Arabia}
\email{shariefd@ksu.edu.sa}

\author[Euripedes C. da Silva]{Euripedes da Silva}
\address[Euripedes C. da Silva]{Instituto Federal de Educa\c c\~ao, Ci\^encia e Tecnologia do Cear\'a, 
	Avenida Parque Central, 1315, CEP 61939-140, Maracana\'u, Cear\' a, Brazil}
\email{euripedescarvalhomat@gmail.com}

\author[Paulo Sousa]{Paulo Sousa}
\address[Paulo Sousa]{Universidade Federal do Piauí, Bairro Ininga, 64049-550, Teresina, Piauí,  Brazil}
\email{paulosousa@ufpi.edu.br}


\begin{abstract}
In this work, we study the geometric properties of spacelike foliations by hypersurfaces on a Lorentz manifold. We investigate conditions for the leaves being stable, totally geodesic or totally umbilical. We consider that $\overline{M}^{n+1}$ is equipped with a timelike closed conformal vector field $\xi$. If the foliation has constant mean curvature, we show that the leaves are stable. When the leaves are compact spacelike hypersurfaces we show that, under certain conditions, its are totally umbilic hypersurfaces. In the case of foliations by complete noncompact hypersurfaces, we using a Maximum Principle at infinity to conclude that the foliation is totally geodesic.

\end{abstract}
\maketitle
\section{Introduction and statement of the main results}

The study of the foliation on Riemannian or Lorentz manifolds have been studied by many authors. In the context Riemannian we point out the works  \cite{lucas, montielriem, barroscaminha, oshikiristable}. They focus on the geometry of the leaves in order to answer if the leaves are totally geodesic, stable and umbilical hypersurfaces. The totally geodesic foliations were studied in  \cite{abe, lucas, David}. We highlight the work of Barbosa et. al in \cite{lucas}, where they proved that a foliation whose leaves have constant mean curvature should be totally geodesic, and such a foliation does not exist in the sphere. The Montiel's analysis of minimality \cite{montielriem}, shows that the existence of a closed conformal vector field implies the existence of an umbilical foliation. The study of the stability of foliations is included in the works of Barbosa, Carmo and Montiel (see \cite{lucasestum, lucasestdois, montielriem}). The notion of stability for Riemannian manifolds was presented by Barbosa and Carmo \cite{lucasestdois}, where they show that the spheres are critical points of the area functional for variations that preserve volume.

A Lorentz manifold $(\overline{M},\overline{g})$ of dimension $(n+1)$ is a smooth manifold endowed with  a pseudo-Riemannian metric $\overline{g}$ of signature $(+ + \cdots + -)$. So, tangent vectors $X_p \in T_p\overline{M}$ are called spacelike, lightlike or timelike if $\overline{g}(X_p,X_p)$ is positive, equal to zero or negative, respectively. We say that a Lorentz manifold $\overline{M}$ is time-oriented if there is a vector field $X \in \mathfrak{X}(\overline{	M})$ such that at all points $p \in \overline{M}$, the vector $X_p$ is timelike. A hypersurface $L$ of a Lorentz manifold $\overline{M}$ is said to be spacelike if the metric on $L$, induced by the metric of $\overline{M}$, is positive definite, that is, the induced metric on $L$ is Riemannian.

 In the context spacelike foliations on a Lorentz manifold we point out the works (see \cite{gervasiopalmas, barbosaoliker, euripedes, montiellorentz}). Barbosa and Oliker \cite{barbosaoliker} considered the same problem for spacelike hypersurfaces with constant mean curvature in the Lorentzian context and they proved that such hypersurfaces are also the critical point of the area functional for spacelike variations that preserve volume. In 1999 Montiel \cite{montiellorentz}, studied spacelike foliations on  Lorentz manifolds and showed that the existence of a timelike closed and conformal vector field on such a manifold furnishes a totally umbilical spacelike foliation. More recently, Chaves and Silva in \cite{euripedes} characterize totally geodesic spacelike foliations on Lorentz manifolds. As a consequence of the main result, they obtained an obstruction to the existence of spacelike foliations on anti de
 Sitter space with hypothesis at infinity.

Taking into account the motivations given above, the following questions appear naturally:

{\it 
\begin{enumerate}[(1)]
	\item For a given foliation by spacelike hypersurfaces on a Lorentz manifold what conditions should be imposed on the leaves to be stable, umbilical or geodesic hypersurfaces?
	\item Is it possible to know or determine the geometry of a foliation knowing information about such a foliation only at infinity? 
\end{enumerate}
}
In this work, we aim to give answers to the above questions. 

In Section 2, we state some preliminaries and basic equations.

In section 3, we present conditions for a constant mean curvature foliation to be stable. More precisely,

{\it Let $\mathcal{F}$ be a spacelike foliation by hypersurfaces on a Lorentz manifold $\overline{M}^{n+1}$. Assume that each leaf of $\mathcal{F}$ has the same constant mean curvature. Then we have:
	\begin{enumerate}
		\item Each leaf of $\mathcal{F}$ is stable.
		\item For a compact leaf $L$ of $\mathcal{F}$. A normal section $V=fN\in \Gamma(\nu(L))$ is a Jacobi field if and only if $\nabla f + f\overline{\nabla}_NN=0$ on $L$, where $\nabla f$ stands for the gradient of $f$. Moreover, the following two conditions are equivalent:
		\begin{enumerate}
			\item $L$ admits a non-trivial Jacobi field.
			
			\item $L$ admits a non-vanishing Jacobi field.
		\end{enumerate}
\end{enumerate}}

The result quoted above is an extension of the results proved by Oshikiri \cite{oshikiristable} and Barbosa et al. \cite{lucasfoli}.

In section 4, we consider spacelike foliations transversal to a unit timelike closed and conformal vector field and under a geometric constraint at infinity we show that one such foliation is totally geodesic. More precisely

{\it Let $(\overline{M}^{n+1},\mathcal{F},\overline{g},N)$ be a foliation of a Lorentz manifold. Assume that $\mathcal{F}$ is transversal to a timelike closed conformal field $\xi$ with constant norm $1$ and oriented by the choice of $N$ such that $ \overline{g}(N,H_{\xi}\cdot\xi)\geq0$. If the second fundamental form $A$ of $L\in \mathcal{F}$ with respect to $N$ is nonnegative and $N$ converges to $\xi$ at infinity, then $\mathcal{F}$ is a totally geodesic foliation.}

We present a relevant particular case of Theorem \ref{FolTolUmb}, a Bernstein-type result, quoted below.

{\it Let $\mathcal{F}$ be a foliation by entire graphics on Euclidian space $\mathbb{R}^{n+1}$. If the second fundamental form of the foliation $\mathcal{F}$ with respect to the upward-pointing unit normal vector field $N$ is nonnegative and $N$ converges to a fixed vector $\xi$ at infinity, then each leaf of $\mathcal{F}$ is a hyperplane orthogonal to $\xi$.}

Finally in section 5, we studying spacelike compact foliations
and assuming certain geometric restrictions we prove that foliation is totally umbilical. For this, we will consider each leaf of foliation immersed as a hypersurface. 
Let $L$ be a compact leaf of the spacelike foliation on a Lorentz manifold $\overline{M}$, we prove that 

{\it Let $\xi $ be a closed conformal vector field on an $(n+1)$-dimensional Lorentz manifold $\left( \overline{M}, \overline{g}\right) $ and $L$ be a spacelike compact and connected hypersurface of $\overline{M}$ with timelike unit normal vector field $N$ such that $\xi $ is not tangent to $L$. Let $w$ be the tangential component of $\xi $ to $L$. If the Ricci curvature $%
	\ric\left( w,w\right) $ of the hypersurface satisfies%
	\[
	\ric\left( w,w\right) \geq \frac{n-1}{n}\left(\Div w\right) ^{2},
	\]%
then $L$ is totally umbilical hypersurface.}


\section{Preliminaries and basic equations}

In this Section, we introduce some basic facts and notations that will appear in the paper.

Let $\overline{M}^{n+1}$ be a Lorentz manifold endowed with a semi-Riemannian metric $\overline{g}$ and  $\mathcal{F}$ is a spacelike foliation of codimension one on $\overline{M}$. Let $N$ be a unit timelike vector field normal to the leaves of $\mathcal{F}$, then the mean curvature of the leaf is exactly the mean curvature on the direction of $N$. Hereafter, a quadrup let $(\overline{M},\mathcal{F},\overline{g},N)$ always means these situations.

For each $p\in\overline M$, we consider the linear 
operator $A:T_{p}\overline M\to T_{p}\overline M$ defined by $A(Y(p))=-\overline{\nabla}_{Y(p)}N$, where $\overline{\nabla}$ denotes the Levi-Civitta connection of $\overline M$. It is clear that if $Y$ is a smooth vector field on $\overline M$, then the same is true of $A(Y)$. Moreover, letting $A_{L}$ denote the second fundamental form of a leaf $L$ of $\mathcal F$, which we will also denote by $A$. Thereby, the mean curvature is  
\begin{equation*}
H=-\frac{1}{n}\mbox{tr}(A).
\end{equation*}

For a given point  $p\in \overline{M}$ we can choose an orthonormal frame $\{e_1,...,e_n,e_{n+1}\}$ defined around  $p$ such that the vectors $e_1,...,e_n$ are tangent to the leaves of  $\mathcal{F}$ and  $e_{n+1}=N$. Such a frame is usually called an adapted frame. The  divergent of a vector field  $V$ is defined  locally over  $M$ by 
\begin{equation*}
\Div(V)=\sum_{A=1}^{n+1}{\epsilon_A\overline{g}(\overline{\nabla}_{e_A}V,e_A)}.
\end{equation*}

For a vector field tangent to the leaves of  $\mathcal{F}$ the divergent along the leaves can be computed by
\begin{equation*}
\Div_L(V)=\sum_{i=1}^{n}{\overline{g}(\overline{\nabla}_{e_i}V,e_i)}.
\end{equation*}
Moreover, the Ricci curvature on the direction  $N$  is
\begin{equation*}
\overline{\ric}(N)=-\sum_{i=1}^{n}{\overline{g}(\overline{R}(N, e_i) N,e_i)}.
\end{equation*}

In \cite{euripedes}, Chaves and Silva found an equation that relates the foliation  with the ambient, more precisely.
\begin{proposition} \label{propequation}
Let $\mathcal{F}$ be a spacelike foliation by hypersurfaces on a Lorentz manifold $\overline{M}$ and let $N$ be a unit field normal to the leaves of $\mathcal{F}$ on some open set $U$ of $\overline{M}$. Then on $U$, we have
\begin{enumerate}
\item $\label{divN} \Div N = nH;$
			
\item $\label{divL} \Div_L(\overline{\nabla}_NN)= nN(H)-\|A\|^2+n\overline{\ric}(N)+\|\overline{\nabla}_NN\|^2;$
			
\item $\label{divX} \Div X = \Div_L (\overline{\nabla}_NN) +\|\overline{\nabla}_NN\|^2,$
\end{enumerate}
where $H$ is the mean curvature on the direction $N$.
\end{proposition}

From now on, we will introduce some definitions and basic results about stability. For more details we recommend the references \cite{barbosaoliker}, \cite{barroscaminha} and \cite{oshikiristable}.

\begin{definition}
Let $(\overline{M},\mathcal{F},\overline{g},N)$ be a foliation and $L$ be any leaf of $\mathcal{F}$. We say that $L$ is stable if the following inequality holds
\[
\mathcal{J}''(0)(f)=\int_{L}{\left( f\Delta f-\|A\|^2f^2+\overline{\ric}(N)f^2 \right)dM}\leq0,
\]
for all $f\in \mathcal{F}(L;\mathbb{R})$, where $\mathcal{F} (L;\mathbb{R})$ is the set of differential functions $f:L\rightarrow \mathbb{R}$ com $f|_{\partial L}=0$ and $\int_L{fdM}=0$, and $\Delta$ is the Laplacian of $(L, \overline{g}_{|L})$.
\end{definition}

Let $L$ be a compact leaf of $\mathcal{F}$. As the normal bundle $\nu(L)$ of $L$ in $\overline{M}$ is trivial, any normal section $V\in\Gamma(\nu(L))$ on $L$ can be expressed uniquely as $V=fN$, with $f\in C^{\infty}(L)$. Define $J: \Gamma(\nu(L)) \rightarrow \Gamma(\nu(L))$ by
\[
J(f) = (\Delta f-f\|A\|^2+\overline{\ric}(N)f)N.
\]

\begin{definition}
We say that a normal section $V\in \Gamma(\nu(L))$ is a Jacobi field if $J(V)=0$. 
\end{definition}

\section{Stablity spacelike foliation on Lorentz manifolds}

In \cite{oshikiristable}, Oshikiri studied Jacobi fields and the stability of leaves of codimension-one minimal foliations on a Riemannian manifold and proved the stability of leaves (Theorem 1). Later, Barbosa et al. \cite{lucasfoli} considered codimension-one foliations of Riemannian manifolds whose leaves have the same constant mean curvature and proved that its leaves are strongly stable (Theorem 3.1). 

In our first result we obtain an extension of the results proved by Oshikiri and Barbosa et al. for foliations on Lorentz manifolds. Moreover answers our last question, that is, CMC spacelike foliation are stable. More precisely, we have the following.

\begin{theorem}\label{Theo1} 
Let $\mathcal{F}$ be a spacelike foliation by hypersurfaces on a Lorentz manifold $\overline{M}^{n+1}$. Assume that each leaf of $\mathcal{F}$ has the same constant mean curvature. Then we have:
\begin{enumerate}
\item Each leaf of $\mathcal{F}$ is stable.
		
\item For a compact leaf $L$ of $\mathcal{F}$. A normal section $V=fN\in \Gamma(\nu(L))$ is a Jacobi field if and only if $\nabla f + f\overline{\nabla}_NN=0$ on $L$, where $\nabla f$ stands for the gradient of $f$. Moreover, the following two conditions are equivalent:
\begin{enumerate}
\item $L$ admits a non-trivial Jacobi field.
			
\item $L$ admits a non-vanishing Jacobi field.
\end{enumerate}
\end{enumerate}
\end{theorem}
\begin{proof}
	
For each leaf $L\in\mathcal{F}$ consider $f\in\mathcal{F} (M,\Bbb{R})$, by item (2) of Proposition \ref{propequation} we must have
\begin{eqnarray*}
\nonumber \left(-\|A\|^2f^2+\overline{\ric}(N)\right)f^2 & = & f^2\left( \Div_L(\overline{\nabla}_NN) -\|\overline{\nabla}_NN\|^2\right)\\
& = &\label{eqstabelinter} \Div_L(f^2\overline{\nabla}_NN)-(\overline{\nabla}_NN)(f^2)-f^2\|\overline{\nabla}_NN\|^2.
\end{eqnarray*}
Then, by definition
\begin{eqnarray*}\label{eqdh}
\mathcal{J}''(0)(f)&=&\int_{L}{\left( f\Delta f-\|A\|^2f^2+\overline{\ric}(N)f^2 \right)dM}\\
& = & \int_L{\left(-\|\nabla f\|^2+\Div_L(f^2 \overline{\nabla}_NN)-(\overline{\nabla}_NN)(f^2)-f^2\|\overline{\nabla}_NN\|^2\right)dM}\\
& = & -\int_L{\left(\|\nabla f\|^2+2f(\overline{\nabla}_NN)(f)+f^2\|\overline{\nabla}_NN\|^2-\Div_L(f^2\overline{\nabla}_NN)\right)dM}\\
& = & -\int_L{\|\nabla f+f\overline{\nabla}_NN\|^2dM}+\int_{L}{\Div_L(f^2\overline{\nabla}_NN)dM}.
\end{eqnarray*}
As the support of $f$ is compact, the second term on the right is null and we have 
\begin{eqnarray}\label{SecVar}
\mathcal{J}''(0)(f)= -\int_L{\|\nabla f+f\overline{\nabla}_NN\|^2dM}.
\end{eqnarray}

Hence we conclude that $\mathcal{J}''(0)(f)\leq0$ and, consequently, the leaf $L$ is stable. Now we will prove the first part of item (2). Let $V=fN$ be a Jacobi field, that is,  $J(V)=0$. Thus we have
\[
\mathcal{J}''(0)(f)=\int_{L}{\overline{g}(JV,V)}dM=0,
\]
according to Equation (\ref{SecVar}) we must have $\nabla f+f\overline{\nabla}_NN=0$.  Now we show the converse. If $\nabla f+f\overline{\nabla}_NN=0$, we multiply this equation by $f$ and take the divergent over $L$ 
\[
f\Delta f+\|\nabla f\|^2+ (\overline{\nabla}_NN) (f^2)+f^2\Div_L(\overline{\nabla}_NN)=0.
\]

Replacing item (2) of Proposition \ref{propequation} in the obtained equality, and using that the leaves have constant mean curvature, we infer
\[
(\Delta f-f\|A\|^2+\overline{\ric}(N)f)f+\|\nabla f+f \overline{\nabla}_NN\|=0,
\]
i.e., $(\Delta f-f\|A\|^2+\overline{\ric}(N)f)f=0$. Then, $J(fN)=0$.

If $L$ admits a non-vanishing Jacobi field, so $L$ admits a non-trivial Jacobi field. Conversely, assume that $L$ admits a non-trivial Jacob field $V=fN$, then by Proposition 3.6 in \cite{barbosaoliker} we have $\Delta f-\|B\|^2f+ \overline{\ric}(N)f=c$, for some constant $c\in \mathbb{R}$. Thereby
\[
\mathcal{J}''(0)(f)=\int_{L}{\overline{g}( J(V),V)}=0,
\]
for all $f\in \mathcal{F}(L,\mathbb{R})$. So, $\nabla f+f\overline{\nabla}_NN=0$. Assume that $f(p)=0$ for some $p\in L$. Now let $q\in L$ be a point of $L$, for some $\epsilon> 0 $, let $\gamma:(-\epsilon, 1+\epsilon)\rightarrow L$ be a smooth curve on $L$ connecting $p$ and $q$ with $\gamma(0)=p$ and $\gamma(1)=q$. Then we have
\[
\overline{g}\left(\nabla f+f\overline{\nabla}_NN, \frac{d\gamma}{d t}\right)=0,
\]
that is, 
\[
\frac{df(\gamma(t))}{dt}+f(\gamma(t))\overline{g}\left(
\overline{\nabla}_NN, \frac{d\gamma}{d t}\right)=0
\]     
on $t\in (-\epsilon,1+\epsilon)$. If we fix $\gamma(t)$, then $f(\gamma(t))$ is a solution of the linear differential equation
\[
\frac{dx}{dt}+x\overline{g}\left(
\overline{\nabla}_NN, \frac{d\gamma}{d t}\right)=0.
\]
As $f(\gamma(0))=0$, we get $f(\gamma(t))=0$ on $t\in (-\epsilon,1+\epsilon)$ by uniqueness of solutions. In particular, $f(q)=0$. As $q$ is any point of $L$, it follows that $f\equiv 0$, contradicting the fact that $V=fN$ is a non-trivial Jacob field. Hence we conclude that $f$ does not vanish.
\end{proof}

The following corollary is an immediate consequence of the Theorem \ref{Theo1}.

\begin{corollary}
	Let $\mathcal{F}$ be a minimal spacelike foliation by hypersurfaces on a Lorentz manifold $\overline{M}^{n+1}$. Then each leaf of $\mathcal{F}$ is stable.   
\end{corollary}

\begin{corollary}
Let $\mathcal{F}$ be a closed spacelike foliation by hypersurfaces on a generalized Robertson–Walker spacetime $M^{n+1}=I\times_{\phi} F^n$. Assume that each leaf of $\mathcal{F}$ has the same constant mean curvature. If the warping function $\phi$ satisfies $\phi''\geq \max\{H\phi',0\}$, then each leaf $L\in \mathcal{F}$ is either maximal or a spacelike slice $L=\{t_0\}\times F^n$, for some $t_0\in I$.
\end{corollary}
\begin{proof}
The result follwing of the Theorem \ref{Theo1} and Theorem 1.1 in \cite{barroscaminha}.
\end{proof}


\section{Totally Geodesic Foliations at Infinity}

In \cite{caminhayuri}, the authors considered complete noncompact hypersurfaces of Riemannian manifolds, transversals to a Killing vector field of the constant norm and with nonnegative second fundamental form and proved that it is totally geodesic. 

Let $(\overline{M}^{n+1},\overline{g})$ be a Riemannian manifold endowed with a non-trivial Killing field $\xi$, it is known that the $n$-distribution 
\[
p\in\overline{M}\longmapsto\mathcal{D}_{\xi}(p)=\{v\in T_p\overline{M}:\,\overline{g}(\xi(p),v)=0\}
\]
in general, does not define a global foliation on $\overline{M}$. On the other hand, Montiel (\cite{montielfoliation} and \cite{montiellorentz}) proved that a field that is closed and conformal determines a codimension one umbilical foliation. More specifically, let $\overline{M}^{n+1}$ be a Riemannian manifold endowed with a non-trivial field $\xi$ which is closed and conformal. Then we have that
\begin{itemize}
	\item the conforming factor is $\sigma=\frac{1}{n+1}{\rm div}\,\xi$;
	
	\item the $n$-dimensional distribution $\mathcal{D}_{\xi}$ determines a codimension one umbilical Riemannian foliation $\mathcal{F}(\xi)$ which is oriented by $\mathcal{N}= \frac{\xi}{\|\xi\|}$. Moreover, the functions $\|\xi\|$ and ${\rm div}\,\xi$ are constant on connected leaves of $\mathcal{F}(\xi)$ and each leaf has constant mean curvaure $H_{\xi}=-\frac{{\rm div}\,\xi}{(n+1)\|\xi\|}=- \frac{\sigma}{\|\xi\|}$.
\end{itemize}

\begin{observation}Similar result is valid in the Lorentzian case, for more details see \cite{montiellorentz}.
\end{observation}

Motivated by these works, our next results provide a characterization (in the cases Riemannian and Lorentzian) of a foliation that is transversal to a closed conformal field. For this, we recall some definitions and basic results.

Let $(\overline{M}^{n+1},\overline{g})$ be a Riemannian manifold. A vector field $\xi$ on $\overline{M}$ is said to be closed and conformal if there is a function $\sigma\in C^{\infty}(\overline{M})$ such that
\begin{equation}\label{eqconformal}
\overline{\nabla}_{X}\xi=\sigma X,\,\,\forall\,X\in \mathfrak{X}(\overline{M}).
\end{equation}
Let's assume that $\xi$ is a non-vanishing closed conformal field of constant norm. Replacing $\xi$ by $\frac{\xi} {\|\xi\|}$, if necessary, we can assume that $\|\xi\|=1$, and we do so hereafter.

If $(\overline{M},\mathcal{F},\overline{g},N)$ is a foliation of $\overline{M}$ transversal to $\xi$, so we can assume that
$g(N,\xi)>0$. We let $\theta:\overline{M}\rightarrow [0,\pi/2)$ denote the acute angle between $N$ and $\xi$ at each point, given by the equality
\[
\overline{g}(N,\xi)=\cos \theta.
\]
We say that $N$ converges to $\xi$ at infinity provided $\theta$ converges to $0$ at infinity.

In the result quoted below we consider a foliation $(\overline{M},\mathcal{F},\overline{g},N)$ that is transversal to a closed conformal field $\xi$, assuming that $N$ converges to $\xi$ at infinity and other geometric constraints we proved that $\mathcal{F}$ is totally geodesic.

\begin{theorem}\label{FolTolUmb}
Let $(\overline{M}^{n+1},\mathcal{F},\overline{g},N)$ be a foliation of a Riemannian manifold. Assume that $\mathcal{F}$ is transversal to a closed conformal field $\xi$ with constant norm $1$ and oriented by the choice of $N$ such that $ g(N,H_{\xi}\cdot\xi)\leq0$. If the second fundamental form $A$ of $L\in \mathcal{F}$ with respect to $N$ is nonnegative and $N$ converges to $\xi$ at infinity, then $\mathcal{F}$ is a totally geodesic foliation.
\end{theorem}
\begin{proof}
Define $f=1-\overline{g}(N,\xi)$ on $\overline{M}$, and note that $f\geq 0$. If $f\equiv 0$, then $N$ is identically to $\xi$. In this case, for all $p\in L\in \mathcal{F}$ and  $u, v\in T_pL$ we have that 
	\begin{eqnarray*}
		\overline{g}(A(u),v) & = & -\overline{g} (\overline{\nabla}_uN,v)\\
		& = & -\overline{g}(\overline{\nabla}_u\xi,v)\\
		& = & \overline{g}(-\sigma Iu,v),
	\end{eqnarray*}
thus $A=-\sigma I$. Proving that $L$ is a totally umbilical foliation. Therefore, we may assume that $f$ does not vanish identically. Let $\xi^{\top}$ stand for the orthogonal projection of $\xi$ on $\mathcal{F}$. Since $\xi$ is closed and conformal vector field, we obtain
	\begin{eqnarray*}
		\overline{g}(\nabla f,\xi^{\top}) & = & \xi^{\top}(f)\\
		& = & -\overline{g}(\nabla_{\xi^{\top}}N,\xi)-g(N, \nabla_{\xi^{\top}}\xi)\\
		& = & \overline{g}(A(\xi^{\top}),\xi^{\top}).
	\end{eqnarray*}
The expressions above, together with the nonnegativity of $A$ give
	\[
	\overline{g}(\nabla f,\xi^{\top})\geq0.
	\]
Now note that
	\begin{eqnarray*}
		\mbox{div}_{\overline{M}}(\xi) & = & \mbox{div}_{\overline{M}}(\xi^{\top})+\mbox{div}_{\overline{M}}(\xi^{\bot})\\
		& = & \mbox{div}_{\mathcal{F}}(\xi^{\top})+g(\nabla_{N}\xi,N) +\mbox{div}_{\mathcal{F}}(\xi^{\bot})\\
		& = & \mbox{div}_{\mathcal{F}}(\xi^{\top})+\sigma-g(\xi,N)
		\mbox{tr}(A).
	\end{eqnarray*}
On the other hand, $\mbox{div}_{\overline{M}}(\xi)= (n+1)\sigma$. Therefore, 
	\[
	\mbox{div}_{\mathcal{F}}(\xi^{\top})=n\sigma+\overline{g}
	(\xi,N)\mbox{tr}(A).
	\]
	
By assumpition we have that $H_{\xi}\cdot g(N,\xi)\leq0$, this fact, together with the choice of $N$ and the $H_{\xi}= -\sigma$, give
	\[
	\overline{g}(\xi,N)>0,\,\, \sigma\geq0\,\, {\rm and}\,\,
	\mbox{div}_{\mathcal{F}}(\xi^{\top})\geq0.
	\]
	
Moreover, since $N$ converges to $\xi$ at infinity, we get that $f$ converges to $0$ at infinity. Then, by [Theorem 2.2, \cite{caminhayuri}] we conclude that $\overline{g}(\nabla f,\xi^{\top})=0$ on each leaf $L\in \mathcal{F}$ and $\mbox{div}_{\mathcal{F}}(\xi^{\top})= 0$ on $L\setminus f^{-1}(0)$. Since $g(N,\xi)>0$ on $L$, we conclude that $\mbox{tr}(A)=0$ and $\sigma=0$ on $L\setminus f^{-1}(0)$, and hence $A=0$ on $L\setminus f^{-1}(0)$.
	
Now, note that $f^{-1}(0)=\{p\in L\in \mathcal{F}; N(p)=\xi(p)\}$. If $p$ is in the interior of $f^{-1}(0)$, then $N=\xi$ in a neighborhood of $p$, whence $A= -\sigma I=0$ in such a neighborhood. In particular, $A=0$ in the interior of $f^{-1}(0)$ and, since it already vanishes in $L\setminus f^{-1}(0)$, we conclude that $A=0$ on all of $L$. Therefore, $L$ is totally geodesic.
\end{proof}

Note that essentially the same arguments presented in the proof of the previous result, with minor adaptations, for example the acute angle between $N$ and $\xi$ is defined using the hyperbolic cosine, allows us to approach the case of a complete noncompact spacelike foliation on a Lorentzian ambient space. Even unnecessarily, we emphasize that we consider the completeness of the leaves and not of the Lorentzian ambient space.

\begin{theorem}
Let $(\overline{M}^{n+1},\mathcal{F},\overline{g},N)$ be a foliation of a Lorentz manifold. Assume that $\mathcal{F}$ is transversal to a timelike closed conformal field $\xi$ with constant norm $1$ and oriented by the choice of $N$ such that $ g(N,H_{\xi}\cdot\xi)\geq0$. If the second fundamental form $A$ of $L\in \mathcal{F}$ with respect to $N$ is nonnegative and $N$ converges to $\xi$ at infinity, then $\mathcal{F}$ is a totally geodesic foliation.
\end{theorem}

We present a relevant particular case of Theorem \ref{FolTolUmb}, a Bernstein-type result, quoted below.

\begin{corollary}
Let $\mathcal{F}$ be a foliation by entire graphics on Euclidian space $\mathbb{R}^{n+1}$. If the second fundamental form of the foliation $\mathcal{F}$ with respect to the upward-pointing unit normal vector field $N$ is nonnegative and $N$ converges to a fixed vector $\xi$ at infinity, then each leaf of $\mathcal{F}$ is a hyperplane orthogonal to $\xi$.
\end{corollary}

We emphasize that the non-negativity hypothesis of the second fundamental form in the last result is essential. In fact, without this hypothesis the result is not valid as we will see in the example below.

\begin{example}
For each $c\in\Bbb{R}$, consider the parameterized surface $\Psi_c:\Bbb{R}^2 \to\Bbb{R}^3$ defined by $\Psi_c(u,v)= (u,u^3+c,v)$. Let $L_c=\Psi(\Bbb{R}^2)$, then $\mathcal{F}:= \{L_c:\,c\in\Bbb{R}\}$ is a foliation of $\Bbb{R}^3$ by entire graphics. It is not difficult to verify that $\mathcal{F}$ is oriented by unit normal vector field 
\[
N=\frac{1}{\sqrt{1+9u^4}}(3u^2,-1,0).
\]
Now, consider the fixed vector $\xi=(1,0,0)$. Note that $N$ converges to $\xi$ at infinity, however the leaves of $\mathcal{F}$ are not hyperplanes orthogonal to $\xi$. This fact occurs because the second fundamental form of foliation does not satisfy the non-negativity hypothesis. In fact, the principal curvatures of the leaves are
\[
k_1=0\,\,\, {\rm and}\,\,\, k_2=\frac{6u}{(1+9u^4)^{3/2}}.
\]	
\end{example}


\section{Umbilical Spacelike Foliation on Lorentz Manifold}

Let $(\overline{M},\overline{g})$ be a $(n+1)$-dimensional Lorentz manifold and $\xi $ be a timelike closed conformal vector field on $\overline{M}$. Let $(\overline{M}^{n+1}, \mathcal{F}, \overline{g},N)$ be a foliation. In this section we will prove umbilicity results for the foliation $\mathcal{F}$, for this we will consider each leaf $L$ of $\mathcal{F}$ immersed as a hypersurface in $\overline{M}$.

Using equation \eqref{eqconformal}, we see that the curvature tensor field $\overline{R}$ of $\overline{M}$ is given by%
\begin{equation}
\overline{R}(U,V)W=U(\sigma )V-V(\sigma )U\text{,\quad }U,V,W\in \mathfrak{X}(\overline{M})\text{.}
\end{equation}%
Consequently, the Ricci tensor of the Lorentz manfold $\overline{\ric}$
satisfies%
\begin{equation}
\overline{\ric}\left( V,\xi \right) =-(n-2)V\left( \sigma \right) \text{.}
\end{equation}%
Now, let $L\in\mathcal{F}$ be spacelike hypersurface of $\overline{M}$. Then we have the following
fundamental equations for the hypersurface $L$
\begin{equation}\label{eqgauss}
\overline{\nabla }_{X}Y=\nabla _{X}Y-g(AX,Y)N\text{,\quad }\overline{\nabla }_{X}N=-AX\text{,\quad }X,Y\in \mathfrak{X}(L)\text{,}
\end{equation}%
where $g$ is the induced Riemannian metric and $\nabla $ is the Levi-Civita connection on the hypersurface $L$. We assume that the conformal vector field $\xi $ is not tangent to the hypersurface and therefore, we have the following
\begin{equation}\label{eqdecomposition}
\xi =w+fN\text{,}
\end{equation}%
where $w\in \mathfrak{X}(L)$ is the tangential projection of $\xi $ on the hypersurface $L$ and $f=-\overline{g}\left( \xi ,N\right) $ is a smooth function on $L$. We denote the restriction of the conformal factor $\sigma $ (see \eqref{eqconformal}) to $L$ by the same letter $\sigma $. Taking covariant derivative in equation \eqref{eqdecomposition} with respect to $X\in \mathfrak{X}(L)$ and using equation \eqref{eqgauss}, we have
\begin{equation}\label{eqsigma}
\nabla_{X} w=\sigma X+fAX\text{,\quad } \grad f=Aw\text{.}
\end{equation}%
Using equation \eqref{eqgauss}, we have
\begin{equation}\label{eqcurvature}
\overline{R}(X,Y)N=-\left( \nabla A\right) (X,Y)+\left( \nabla A\right) (Y,X)
\text{,\quad }X,Y\in \mathfrak{X}(L)\text{,}
\end{equation}%
where the covariant derivative $\left( \nabla A\right) (X,Y)=\nabla
_{X}AY-A\left( \nabla _{X}Y\right) $.

Using a normal coordinates frame $\left\{e_{1},..., e_{n} \right\} $ on $L$, we have for $X\in \mathfrak{X}(L)$%
\[
nX\left( H \right) =\sum\limits_{i=1}^{n}g\left( \left( \nabla
A\right) (X,e_{i}),e_{i}\right) 
\]%
and using equation \eqref{eqcurvature}, we have
\begin{eqnarray}\label{eqmeancurvature}
nX\left( H \right) &=&\sum\limits_{i=1}^{n}g\left( \left(\nabla A\right) (e_{i},X)-\overline{R}(X,e_{i})N, e_{i}\right)\\ 
&=&\overline{\ric}\left( X,N\right) +\sum\limits_{i=1}^{n} g\left( \left( \nabla A\right) (e_{i},e_{i}),X\right) \text{.} \nonumber
\end{eqnarray}
Next, we compute $\Div Aw$ using equation \eqref{eqsigma} to arrive at%
\[
\Div Aw=\sum\limits_{i=1}^{n}g\left( \nabla _{e_{i}}Aw,e_{i}\right)
=\sum\limits_{i=1}^{n}g\left( \left( \nabla A\right) (e_{i},w)+A\left(
\sigma e_{i}+fAe_{i}\right) ,e_{i}\right) 
\]%
and using symmetry of the shape operator we have%
\[
\Div Aw=n\sigma H +f\left\Vert A\right\Vert
^{2}+\sum\limits_{i=1}^{n}g\left( \left( \nabla A\right)
(e_{i},e_{i}),w\right) \text{.} 
\]%
Using equation \eqref{eqmeancurvature} in above equation, we have%
\begin{equation}
\Div Aw=n\sigma \alpha +f\left\Vert A\right\Vert ^{2}+nw\left( H
\right) -\overline{\ric}\left( w,N\right) \text{.}
\end{equation}%
Thus, equation \eqref{eqsigma} implies%
\[
\Delta f=n\sigma H +f\left\Vert A\right\Vert ^{2}+nw\left( H
\right) -\overline{\ric}\left( w,N\right) 
\]%
that is,%
\begin{equation}
f\Delta f=n\sigma fH +f^{2}\left\Vert A\right\Vert ^{2}+nfw\left(
H \right) -f\overline{\ric}\left( w,N\right) \text{.}
\end{equation}%
Now, suppose the hypersurface $L$ is compact, on integrating the above
equation, we have:

\begin{lemma}
Let $L$ be a compact spacelike hypersurface of an $(n+1)$-dimensional Lorentz manifold $\left( \overline{M}, \overline{g}\right) $. Then the following holds%
	\[
	\int\limits_{L}f^{2}\left( \left\Vert A\right\Vert ^{2}-nH^{2}\right)
	=\int\limits_{L}\left[ -\left\Vert Aw\right\Vert ^{2}+f\overline{\ric}\left(
	w,N\right) -n\sigma fH -nfw\left( H \right) -nf^{2}H^{2}%
	\right] \text{.} 
	\]
\end{lemma}

Now, using equation \eqref{eqsigma}, we have%
\[
\nabla _{X}\nabla _{Y}w=X(\sigma )Y+\sigma \nabla _{X}Y+X(f)AY+f\left(
\nabla A\right) (X,Y)+fA\left( \nabla _{X}Y\right) 
\]%
and it amount to the following expression for the curvature tensor of the hypersurface $L$.%
\[
R\left( X,Y\right) w=X(\sigma )Y-Y(\sigma )X+X(f)AY-Y(f)AX+f\left( \left(
\nabla A\right) (X,Y)-\left( \nabla A\right) (Y,X)\right) 
\]%
and inserting equation \eqref{eqcurvature} in above equation we conclude%
\[
R\left( X,Y\right) w=X(\sigma )Y-Y(\sigma )X+X(f)AY-Y(f)AX-f\overline{R}%
(X,Y)N\text{.}
\]%
Above equation implies the following expression for the Ricci tensor%
\[
\ric \left( Y,w\right) =-(n-1)Y(\sigma )+g\left( A\left( \grad f\right)
,Y\right) -nH Y(f)-f\overline{\ric}\left( Y,N\right) \text{,}
\]%
that is,%
\[
\ric\left( w,w\right) =-(n-1)w(\sigma )+g\left( A\left( \grad f\right)
,w\right) -nH w(f)-f\overline{\ric}\left( w,N\right) \text{,}
\]%
which in view of equation \eqref{eqsigma} implies%
\begin{equation}
\ric(w,w)=-(n-1)w(\sigma )+\left\Vert Aw\right\Vert ^{2}-nH w(f)-f%
\overline{\ric}\left( w,N\right) \text{.}
\end{equation}%
Thus, we have%
\[
-\left\Vert Aw\right\Vert ^{2}+f\overline{\ric}\left( w,N\right)
=-(n-1)w(\sigma )-nH w(f)-\ric(w,w)
\]%
and inserting this value in Lemma 1, we have%
\begin{eqnarray*}
&&\int\limits_{L}f^{2}\left( \left\Vert A\right\Vert ^{2}-nH^{2}\right)=\\
&&\int\limits_{L}\left[ -(n-1)w(\sigma )-nH w(f)-\ric(w,w)-n\sigma
fH -nfw\left( H \right) -nf^{2}H^{2}\right],\\
\end{eqnarray*}
that is,%
\[
\int\limits_{L}f^{2}\left( \left\Vert A\right\Vert ^{2}-nH^{2}\right)
=\int\limits_{L}\left[ -(n-1)w(\sigma )-nw(H f)-\ric(w,w)-n\sigma
fH -nf^{2}H^{2}\right] \text{.}
\]%
Note that equation \eqref{eqsigma} implies $\Div\left( w\right) =n(\sigma +fH )$ and $\Div \left( \sigma w\right) = n\sigma (\sigma +fH) + w(\sigma )$ and inserting it in above integral we get%
\begin{eqnarray*}
&&\int\limits_{L}f^{2}\left( \left\Vert A\right\Vert ^{2}-nH^{2}\right)=\\
&&\int\limits_{L}\left[ n(n-1)\sigma (\sigma +fH )-nw(H
f)-\ric(w,w)-n\sigma fH -nf^{2}H^{2}\right] \text{.}	
\end{eqnarray*}

Furthermore, we have $\Div\left( H fw\right) =w\left( H
f\right) +nH f(\sigma +fH )$ and inserting it in above equation, yields%
\begin{eqnarray*}
&&\int\limits_{L}f^{2}\left( \left\Vert A\right\Vert ^{2}-nH^{2}\right)=\\
&&\int\limits_{L}\left[ n(n-1)\sigma (\sigma +fH )+n^{2}H f\left(\sigma +fH \right) -\ric(w,w)-n\sigma fH -nf^{2}H^{2}\right] \text{,}
\end{eqnarray*}
that is,%
\begin{equation}\label{eqintegral}
\int\limits_{L}f^{2}\left( \left\Vert A\right\Vert ^{2}-nH^{2}\right)
=\int\limits_{L}\left[ n(n-1)(\sigma +fH )^{2}-\ric(w,w)\right] \text{.}
\end{equation}

Now, we are ready to prove:

\begin{theorem}
Let $\xi $ be a closed conformal vector field on an $(n+1)$-dimensional Lorentz manifold $\left( \overline{M}, \overline{g}\right)$ and $(\overline{M}^{n+1}, \mathcal{F}, \overline{g},N)$ a foliation such that $\xi $ it is not tangent to the leaves $L\in\mathcal{F}$. If the Ricci curvature $\ric\left(\xi^{\top}, \xi^{\top}\right) $ of each leaf $L$ satisfies%
	\[
	\ric\left( \xi^{\top}, \xi^{\top}\right) \geq \frac{n-1}{n}\left(\Div \xi^{\top}\right) ^{2},
	\]%
then $L$ is a totally umbilical hypersurface.
\end{theorem}

\begin{proof}
Using equation \eqref{eqintegral}, we have%
	\begin{equation}
	\int\limits_{L}f^{2}\left( \left\Vert A\right\Vert ^{2}-nH^{2}\right)
	=\int\limits_{L}\left[ \frac{(n-1)}{n}(\Div \xi^{\top} )^{2}-\ric(\xi^{\top}, \xi^{\top})\right] 
	\text{.}
	\end{equation}%
	Using the condition in the statement, we have%
	\[
	\int\limits_{L}f^{2}\left( \left\Vert A\right\Vert ^{2}-nH^{2}\right)
	\leq 0\text{,}
	\]%
	which in view of Schwartz's inequality implies%
	\[
	f^{2}\left( \left\Vert A\right\Vert ^{2}-nH^{2}\right) =0\text{.}
	\]%
	Note that $\xi $ is not tangent to $L$ imples $f\neq 0$ and above equation on connected $L$ implies $\left\Vert A\right\Vert ^{2}=nH^{2}$. This equality holds if and only if $A=H I$. Thus, $L$ is totally umbilcal hypersurface.
\end{proof}

\noindent {\bf Acknowledgments.} The third author would like to thank Edson Sampaio for his interest in this manuscript and for listening patiently to my first ideas in this research, in particular, in the discussion on Example 4.5. The fourth author was partially supported by CNPq, Grant 402668/2016-2.

\end{document}